\documentclass[a4paper]{article}
\usepackage{amssymb}
\usepackage{graphicx}
\usepackage{amsmath}
\usepackage{amsfonts}
\usepackage{amsthm}
\usepackage{mathrsfs}
\usepackage{dsfont}
\usepackage{indentfirst}
\usepackage{esint}
\usepackage{hyperref}
\usepackage{yhmath}

\numberwithin{equation}{section}

\newcommand{\osc}{\mathrm{osc}}

\newcommand{\Span}{\mathrm{Span}}
\newcommand{\Z}{\mathds{Z}}

\title{\Large \bf \boldmath\ \\ Polynomial Growth Harmonic Functions on Finitely Generated Abelian Groups} 

\author{\large  Bobo Hua$^\ast$\ \ \ \  J\"{u}rgen Jost$^{\dag}$ \ \ \ \ Xianqing Li-Jost$^{\ddag}$} 

\date{}

\par
\begin{document}

\maketitle

\renewcommand{\thefootnote}{\fnsymbol{footnote}}

\footnotetext{\hspace*{-5mm} \begin{tabular}{@{}r@{}p{13.4cm}@{}}
$^\ast$ $^{\dag}$ $^{\ddag}$ & Max Planck Institute for Mathematics in the Sciences, Leipzig, 04103, Germany.\\
$^{\dag}$ & supported by ERC Advanced Grant FP7-267087.\\
&{Email: bobohua@mis.mpg.de, jost@mis.mpg.de, xli-jost@mis.mpg.de}\\
&The Mathematics Subject Classification 2010: 31C05, 05C63, 82B41.
\end{tabular}}

\renewcommand{\thefootnote}{\arabic{footnote}}

\newtheorem{theorem}{Theorem}[section]
\newtheorem{con}[theorem]{Conjecture}
\newtheorem{lemma}[theorem]{Lemma}
\newtheorem{corollary}[theorem]{Corollary}
\newtheorem{definition}[theorem]{Definition}
\newtheorem{conv}[theorem]{Assumption}
\newtheorem{remark}[theorem]{Remark}

\bigskip

\begin{abstract}  In the present paper, we develop geometric analysis techniques on Cayley graphs of finitely generated abelian groups
to study the polynomial growth harmonic functions. We provide a
geometric analysis proof of the classical Heilbronn theorem \cite{H}
and the recent Nayar theorem \cite{N} on polynomial growth harmonic
functions on lattices $\mathds{Z}^n$ that does not use a
representation formula for harmonic functions. In the abelian group
case, by Yau's gradient estimate we actually give a simplified proof
of a general polynomial growth harmonic function theorem of
Alexopoulos \cite{Al1}. We calculate the precise dimension of the
space of polynomial growth harmonic functions on finitely generated
abelian groups by linear algebra, rather than by Floquet theory
Kuchment-Pinchover \cite{KP}. While the Cayley graph not only
depends on the abelian group, but also on the choice of a generating
set, we find that this dimension depends only on the group itself.
Moreover, we also calculate the dimension of solutions to higher
order Laplace operators.
\medskip
\end{abstract}
\section{Introduction}
Classically, in 1948 Heilbronn \cite{H} proved the polynomial growth
harmonic function theorem on the lattice $\mathds{Z}^n$ that
polynomial growth discrete harmonic functions are polynomials, and
calculated the dimension of the space of polynomial growth discrete
harmonic functions. Recently, Nayar \cite{N} gave another proof of
this theorem by the probabilistic method. Their proofs all depend on
the representation formula for discrete harmonic functions. In this
paper, we give a geometric analysis proof that can yield more
general results.

The study of harmonic polynomials on $\mathds{R}^n$ is classical,
and the precise dimension calculation of harmonic polynomials can be
found in \cite{L2}. In 1975, Yau \cite{Y1} proved the Liouville
theorem for harmonic functions on Riemannian manifolds with
nonnegative Ricci curvature. Then Cheng-Yau \cite{CheY} used
Bochner's technique to derive the gradient estimate for positive
harmonic functions, called Yau's gradient estimate, which implies
that sublinear growth harmonic functions on these manifolds are
constant. Then Yau \cite{Y2,Y3} conjectured that the space of
polynomial growth harmonic functions of growth order less than or
equal to $d$ on Riemannian manifolds with nonnegative Ricci is of
finite dimension. Li-Tam \cite{LT} and Donnelly-Fefferman \cite{DF}
independently solved the conjecture for manifolds of dimension two.
Then Colding-Minicozzi \cite{CM1,CM2,CM3} gave the affirmative
answer by using the volume doubling property and the Poincar\'e
inequality for arbitrary dimension. The simplified argument by the
mean value inequality can be found in \cite{L1,CM4} where the
dimension estimate is asymptotically optimal. This inspired many
generalizations on manifolds \cite{W,T,STW,LW1,LW2,CW,KLe,Le} and on
singular spaces \cite{D1,Kl,Hu1,Hu2,HJL}. In this paper, we give the
precise dimension calculation of polynomial growth harmonic
functions on finitely generated abelian groups.

In another direction, Avellaneda-Lin \cite{AvL} first proved the
polynomial growth harmonic function theorem for elliptic
differential operators with periodic coefficients in $\mathds{R}^n$
(see \cite{MS,Al3,Al2,ALO,Lin,LW3} for more generalizations).
Alexopoulos \cite{Al1} proved a more general polynomial growth
harmonic function theorem for groups of polynomial volume growth. By
a celebrated theorem of Gromov \cite{G}, every finitely generated
group $G$ of polynomial growth is virtually nilpotent (i.e. it has a
nilpotent subgroup $H$ of finite index). For some torsion-free
finite-index subgroup $H'$ of $H,$ it can be embedded as a lattice
in a simply connected nilpotent Lie group $N.$ By considering the
exponential coordinate of $N,$ $N$ is identified with $\mathds{R}^n$
(see \cite{Ra}). Alexopoulos proved that every polynomial growth
harmonic function on $G,$ when restricted to $H',$ can be extended
to a polynomial on $\mathds{R}^n.$ He used the homogenization
theory, Krylov-Safonov's argument and some ideas of Avellaneda-Lin
\cite{AvL}. Instead of doing that, we shall prove Yau's gradient
estimate on abelian groups to simplify the argument in this special
case.

From another point of view, Kuchment-Pinchover \cite{KP} introduced
a method in Floquet theory to study the space of solutions to
periodic elliptic equations on the abelian cover of a compact
Riemannian manifold or a finite graph. They calculated the dimension
of polynomial growth solutions to the standard operators, in
particular, the Laplace operators, which covers our results. The aim
of this paper is to provide a geometric point of view of this
problem. We will use some geometric analysis favored methods and
basic linear algebra to investigate this problem which may shed some
light on how to solve these problems on general graphs.

We develop some geometric analysis techniques on Cayley graphs of
finitely generated abelian groups. Firstly, we prove Yau's gradient
estimate (see Theorem \ref{YGE}) for positive discrete harmonic
functions. Note that Kleiner \cite{Kl} obtained the Poincar\'e
inequality on Cayley graphs of finitely generated (not necessarily
abelian) groups (see also \cite{ST}). For the abelian case,
combining it with the natural volume doubling property, we obtain
the uniform Poincar\'e inequality (see Lemma \ref{PI}) from which
the mean value inequality follows by the Moser iteration. In
addition, for an abelian group, Bochner's formula can be easily
verified, i.e. $|\nabla u|^2$ is subharmonic for a discrete harmonic
function $u$ which is essentially due to Chung-Yau \cite{CY} and
Lin-Yau \cite{LY} on Ricci flat graphs. Then these results together
imply Yau's gradient estimate on Cayley graphs of finitely generated
abelian groups as in the case of Riemannian manifolds with
nonnegative Ricci curvature. Besides the geometric method, this can
also be obtained by the local central limit theorem, a probabilistic
method (see Theorem 1.7.1 in Lawler \cite{La} for the standard
lattice case). By the fundamental theorem of finitely generated
abelian group (see \cite{Su,R}), any finitely generated abelian
group $G$ is isomorphic to the direct sum
$\mathds{Z}^m\bigoplus\oplus_{i=1}^l\mathds{Z}_{q_i},$ where $m,l\in
\mathds{N}$, $q_i=p_i^{a_i}$ for some prime number $p_i$ and
$a_i\in\mathds{N}.$

\begin{theorem}[Yau's gradient estimate]\label{YGE}Let $(G,S)$ be the Cayley graph of a finitely generated abelian group with symmetric generating set $S$ $(i.e. \ S=-S),$ and $G\cong\mathds{Z}^m\bigoplus\oplus_{i=1}^l\mathds{Z}_{q_i}.$ Then there exist constants $C_1$ and $C_2$ depending only on $m$ and $S,$ such that for any $x\in G,$ $R\geq 1$ and any positive discrete harmonic function $f$ on $B_{C_1R}(x),$ we have
\begin{equation}\label{YGE1}|\nabla f|(x)\leq \frac{C_2}{R}f(x).\end{equation}
\end{theorem}

Secondly, from Yau's gradient estimate, we know that $|\nabla f|\leq
CR^{d-1},$ on $B_R$ for $R\geq R_0$ if the discrete harmonic function
$f$ satisfies $|f|\leq CR^d$ on $B_R$ for $R\geq R_1.$ That is, the
growth order decreases when we take derivatives.

This is the key to  an induction argument to give a geometric
analysis proof of Heilbronn's theorem. This scheme will, in fact,
work  for all abelian groups. Let us denote the set of polynomial
growth harmonic functions of growth order less than or equal to $d$
on the Cayley graph $(G,S)$ by $H^d(G,S):=\{u:G\rightarrow
\mathds{R}\mid\ L^Su=0, |u|(x)\leq C(d^S(p,x)+1)^d\},$ where $L^S$
is the Laplacian operator on $(G,S),$ $d^S$ is the distance function
to some fixed $p\in G.$ For the finitely generated abelian group
$G=G_1\oplus
G_2\cong\mathds{Z}^m\bigoplus\oplus_{i=1}^l\mathds{Z}_{q_i},$ we
define the natural projection $$\pi_{G_1}:\ G\rightarrow G_1,$$
$$x\mapsto \pi_{G_1}(x)=x_1,$$ for $x=x_1+x_2,$ $x_1\in G_1$ and
$x_2\in G_2.$ It is easy to see that $\pi_{G_1}S$ is a generating
set of $G_1$ if $S$ is a generating set of $G$.

\begin{theorem}[Generalized Heilbronn's theorem]\label{GHT} Let $(G,S)$ be the Cayley graph of a finitely generated abelian group, $G=G_1\oplus G_2\cong\mathds{Z}^m\bigoplus\oplus_{i=1}^l\mathds{Z}_{q_i}.$  Then $$H^d(G,S)=H^d(G_1,\pi_{G_1}S),$$ moreover any $f\in H^d(G,S)$ is a polynomial when it is restricted to $G_1\cong \mathds{Z}^m,$ and it is constant on $G_2$ i.e. $f(x+w)=f(x),$ for $\forall x\in G,$ $\forall w\in G_2.$
\end{theorem}

Nayar \cite{N} proved a strong version of Heilbronn's theorem. We
denote by $HM^d(G,S):=\Span\{u:G\rightarrow \mathds{R}\mid\ L^Su=0,
u(x)\geq -C(d^S(p,x)+1)^d \}$ the linear span of one-sided bounded
polynomial growth harmonic functions. It is trivial that
$H^d(G,S)\subset HM^d(G,S).$ By the Harnack inequality
(\ref{HNKI1}), we give a geometric analysis proof of Nayar's
theorem.
\begin{theorem}[Generalized Nayar's theorem]\label{GNT} Let $(G,S)$ be the Cayley graph of a finitely generated abelian
group, $G=G_1\oplus
G_2\cong\mathds{Z}^m\bigoplus\oplus_{i=1}^l\mathds{Z}_{q_i}$. Then
$$HM^d(G,S)=H^d(G,S)=H^d(G_1,\pi_{G_1}S).$$

\end{theorem}

Thirdly, for discrete harmonic polynomials we calculate the precise
dimension by linear algebra. In fact, instead of using the technical
lemma (see \cite{DS,Hob}) from difference equations as Heilbronn
\cite{H} did in the lattice case, we apply the dimension comparison
argument with harmonic polynomials in $\mathds{R}^n.$ Conversely,
our argument provides a proof of this difference equation lemma.
Moreover, we can calculate the dimension of the space of polynomial
growth harmonic functions on arbitrary Cayley graph of a finitely
generated abelian group. It is surprising that the dimension of
polynomial growth harmonic functions does not depend on the choice
of generating set $S$ for the abelian group $G.$ Actually, the graph
structures of two Cayley graphs of the abelian group $G$ with two
generating set $S_1$, $S_2$ can be quite different. While the
Laplacian operator depends on the generating set, the dimension of
polynomial growth harmonic functions does not. We denote by
$HP^k(\mathds{R}^m)$ the space of harmonic polynomials on
$\mathds{R}^m$ with degree less than or equal to $k,$
$k\in\mathds{N}\cup\{0\}$. It is well known (see \cite{L2}) that
$$\dim HP^k(\mathds{R}^m)=\binom{m+k-1}{k}+\binom{m+k-2}{k-1}.$$

\begin{theorem}[Dimension calculation]\label{DCT}Let $(G,S)$ be the Cayley graph of a finitely generated abelian group, $G=G_1\oplus G_2\cong\mathds{Z}^m\bigoplus\oplus_{i=1}^l\mathds{Z}_{q_i}.$ Then
$$\dim HM^d(G,S)=\dim H^d(G,S)=\dim HP^{[d]}(\mathds{R}^m).$$
\end{theorem}

In the last section, we study some higher order operators. The $n$th
order Laplace operator is defined as
$$L^{n,S}=L^S\circ
L^S\circ\cdots\circ L^S,$$ i.e. $n$-times composition of Laplace
operators. These operators correspond to higher order elliptic
operators in the continuous setting. For instance, when $n=2,$
$L^{2,S}$ is a discrete generalization of the bi-Laplace operator
$\Delta^2.$ A function $u$ on $G$ is called discrete $n$-harmonic if
$L^{n,S}u=0.$ Let us denote by $H^{n,d}(G,S):=\{u:G\rightarrow
\mathds{R}\mid\ L^{n,S}u=0, |u|(x)\leq C(d^S(p,x)+1)^d\}$ the space
of polynomial growth $n$-harmonic functions of growth order not
larger than $d$. By the techniques developed before for harmonic
functions, we also obtain the precise dimension of $H^{n,d}(G,S).$
\begin{theorem}\label{hoo} Let
$(G,S)$ be the Cayley graph of a finitely generated abelian group,
$G=G_1\oplus
G_2\cong\mathds{Z}^m\bigoplus\oplus_{i=1}^l\mathds{Z}_{q_i}.$  Then
$$H^{n,d}(G,S)=H^{n,d}(G_1,\pi_{G_1}S),$$ moreover any $f\in H^{n,d}(G,S)$ is a
polynomial on $G_1\cong \mathds{Z}^m,$ and constant on $G_2.$ In
addition,
\begin{equation*}
\dim
H^{n,d}(G,S)=\sum_{i=[d]-2n+1}^{[d]}\binom{m+i-1}{i}.\end{equation*}
\end{theorem}

Since we do not use the representation formula, this method can be
applied in more general settings. In fact, Bochner's formula is
proved for Ricci flat graphs by Chung-Yau \cite{CY}, but it is not
easy to get the volume control or the Poincar\'e inequality in that
case. Lin-Yau \cite{LY} proved Bochner's formula for general graphs,
but this does not lead to a version of  Yau's gradient estimate
analogous to the case of Riemannian manifolds with nonnegative Ricci
curvature. In the Cayley graph case, Kleiner \cite{Kl} obtained the
Poincar\'e inequality, but for non-abelian groups, Bochner's formula
is unavailable (consider the free group case). For some special
graphs which can be embedded into a surface with nonnegative
sectional curvature in the sense of Alexandrov, Hua-Jost-Liu
\cite{HJL} proved the volume doubling property and the Poincar\'e
inequality but Bochner's formula. It seems that Bochner's formula is
sensitive to the local structure, but the volume growth property and
the Poincar\'e inequality are not, c.f. \cite{CS}. Hence the abelian
groups are very suitable candidates for the application of Bochner's
formula and Yau's gradient estimate. In addition, Alexopoulos'
theorem \cite{Al1} is more general than ours, but it depends on the
embedding of the nilpotent subgroup to simply connected Lie group.
It seems hard to calculate the precise dimension of polynomial
growth harmonic functions on groups of polynomial volume growth. One
step back, we give a dimension estimate by the geometric method in
\cite{HJ}.

\section{Preliminaries and Notations}
Let $G$ be an abelian group. It is called finitely generated if it
has a finite generating set. In this paper, we assume that any
finite generating set $S=\{s_1,s_2,\cdots,s_{2l}\}$ of $G$ is
symmetric, i.e. $S=-S$, or more precisely $s_i=-s_{i+l}, 1\leq i\leq
l$, but we do allow that elements of $S$ are repeated, that is
possibly $s_i=s_j$ for some $i\neq j.$ We also allow $0\in S.$ For
any finitely generated abelian group $G$ with a generating set $S,$
we have the associated Cayley graph $(V,E)$ for which $V=G,$ and
$xy\in E$ (denoted by $ x\sim y$) if $y-x\in S$ for $x,y\in V.$ The
duplicity of elements in $S$ produces multiedges between vertices,
and $0\in S$ makes self-loops. The vertices $x$ and $y$ are called
neighbors if $x\sim y.$  The degree of a vertex $x$ is the number of
its neighbors. Note that all vertices in $(G,S)$ have the same
degree $\sharp S,$ the cardinality of $S$. For the lattice
$\mathds{Z}^n$ with the standard generating set
$S^0=\{e_i\}_{i=1}^{2n},$ where $e_i=(0,\cdots,1,\cdots,0)$ is the
$i$-th unit vector and $e_i=-e_{i+n},$ $1\leq i\leq n$, we obtain
the standard integer lattice in $\mathds{R}^n$. This is the object
Heilbronn \cite{H} and Nayar \cite{N} studied. In this paper, we
consider the discrete harmonic functions on Cayley graphs of $G$
with arbitrary finite generating set $S.$

The Cayley graph of $(G,S)$ is endowed with a natural metric, called
the word metric (c.f. \cite{BBI}). For any $x,y\in G,$ the distance
between them is defined as the length of the shortest path
connecting $x$ and $y$ (each edge is of length one),
$$d^S(x,y):=\inf\{k\in \mathds{N}\mid \exists\  x=x_0\sim x_1\sim\cdots \sim x_k=y\}.$$
Denote by $B^S_r(x):=\{y\in G\mid d^S(y,x)\leq r\}$ the closed
geodesic ball centered at $x$ of radius $r$ $(r>0).$ The volume of
$B^S_r(x)$ is $|B^S_r(x)|:=\sharp G\cap B^S_r(x),$ i.e. the number
of vertices contained in $B^S_r(x).$ For the subset $\Omega\subset
G,$ $d^S(x,\Omega):=\inf\{d^S(x,y)\mid y\in \Omega\}$ for any $x\in
G,$ $\partial\Omega:=\{z\in G\mid d^S(z,\Omega)=1\},$ and
$\bar{\Omega}:=\Omega\cup\partial\Omega.$ For any function
$f:\bar{\Omega}\rightarrow \mathds{R},$ the discrete Laplacian
operator is defined on $\Omega$ as $(x\in \Omega)$
$$L^Sf(x)=\sum_{y\sim x}(f(y)-f(x)).$$ The function $f$ is called discrete harmonic (subharmonic) on $\Omega$ if $L^Sf(x)=0\ (\geq0)$ for all $x\in\Omega.$
In the lattice case $(\mathds{Z}^n, S^0)$, the Laplacian operator is
$$L^{S^0}f(x)=\sum_{i=1}^{2n}(f(x+e_i)-f(x)).$$ Moreover, the $n$th order Laplace operator is defined as
$$L^{n,S}=L^S\circ
L^S\circ\cdots\circ L^S,$$ i.e. $n$-times composition of Laplace
operators. These operators are counterparts of higher order elliptic
operators in the continuous setting. A function $f$ is called
$n$-harmonic if $L^{n,S}f=0.$ The gradient of $f$ at $x\in \Omega$
is defined as $|\nabla^S f|(x)=\sqrt{\sum_{y\sim x}(f(y)-f(x))^2}.$
We also need the partial difference operator (for $s\in S$ and $x\in
\Omega$)
$$\delta_sf(x):=f(x+s)-f(x).$$

For Cayley graphs of $(G, S_1)$ and $(G, S_2)$, it is known (c.f. \cite{M,G}) that
$$C_1d^{S_1}(x,y)\leq d^{S_2}(x,y)\leq C_2d^{S_1}(x,y),$$ for any $x,y\in G,$ where $C_1$ and $C_2$ depend only on $S_1, S_2.$ They are bi-Lipschitz equivalent in the metric point of view. Hence for any $x\in G,$ $r>0,$ $$|B^{S_1}_{\frac{r}{C_2}}(x)|\leq|B^{S_2}_r(x)|\leq|B^{S_1}_{\frac{r}{C_1}}(x)|.$$

By the fundamental theorem of finitely generated abelian groups
\cite{Su,R}, any finitely generated abelian group $G$ is isomorphic
to the direct sum
$\mathds{Z}^m\bigoplus\oplus_{i=1}^l\mathds{Z}_{q_i},$ where $m,l\in
\mathds{N}$, $q_i=p_i^{a_i}$ for some prime number $p_i$ and some
$a_i\in\mathds{N}.$ Hence there exists a generating set
$S^0=\{e_1,\cdots,e_{2m}, w_1,\cdots,w_{2l}\}$ ($e_i=-e_{i+m},$
$w_j=-w_{j+l}$, for $1\leq i\leq m,$ $1\leq j\leq l$) such that $G$
is identified with
$\mathds{Z}^m\bigoplus\oplus_{i=1}^l\mathds{Z}_{q_i},$ where
$\{e_1\cdots e_{2m}\}$ generates the torsion-free part
$\mathds{Z}^m$ and $\{w_1\cdots w_{2l}\}$ generates the torsion part
$\oplus_{i=1}^l\mathds{Z}_{q_i}$. For $(G,S^0)$ it is easy to see
that $C_1(m,S^0)r^m\leq|B_r^{S^0}(x)|\leq C_2(m, S^0)r^m,$ and
$|B_{2r}^{S^0}(x)|\leq C_3(m,S^0)|B^{S^0}_r(x)|,$ for any $x\in G$
and $r\geq1.$ Hence by the bi-Lipschitz equivalence, for any Cayley
graph $(G,S)$ we have
$$C_1(m,S^0,S)r^m\leq|B_r^{S}(x)|\leq C_2(m,S^0,S)r^m,$$
\begin{equation}\label{VDP}|B^{S}_{2r}(x)|\leq C_3(m,S^0,S)|B^S_r(x)|.
\end{equation} The volume growth property (\ref{VDP}) is called the volume doubling property.

In the sequel, for simplicity we shall omit $S$ from our notation,
e.g.,  $B_r(x):=B^S_r(x)$, if it does not cause any confusion. Also
harmonic functions on $G$ mean discrete harmonic functions. And
constants $C$ may change from line to line.

\section{Yau's Gradient Estimate}
In this section, we prove Bochner's formula on finitely generated
abelian groups and derive Yau's gradient estimate analogous to the
one in Riemannian manifolds with nonnegative Ricci curvature.

Kleiner \cite{Kl} proved the Poincar\'e inequality on Cayley graphs of finitely generated (not necessarily abelian) groups. Let $(G,S)$ be the Cayley graph of the finitely generated abelian group \begin{equation}\label{ABEL}G\cong\mathds{Z}^m\bigoplus\oplus_{i=1}^l\mathds{Z}_{q_i}.\end{equation} By the volume doubling property (\ref{VDP}), we obtain the uniform Poincar\'e inequality.
\begin{lemma}[Poincar\'e inequality, \cite{Kl}]\label{PI} Let $(G,S)$ be the Cayley graph as (\ref{ABEL}). Then there exists a constant $C(m,S)$ such that for any $p\in G,$ $R>0$ and any function $f: B_{3R}(p)\rightarrow\mathds{R}$ we have
\begin{equation}\label{PI1}\sum_{x\in B_R(p)}(f(x)-f_{B_R})^2\leq CR^2\sum_{x,y\in B_{3R}(p);x\sim y}(f(x)-f(y))^2,\end{equation} where $f_{B_R}=\frac{1}{|B_R(p)|}\mathop{\sum}_{x\in B_R(p)}f(x).$
\end{lemma}

Note that by the independent works of Delmotte \cite{D2} and
Holopainen-Soardi \cite{HS} the Moser iteration can be carried out
for harmonic functions on graphs satisfying the volume doubling
property and the Poincar\'e inequality. First, the Caccioppoli
inequality for harmonic functions was obtained for general graphs
with bounded degree (c.f. \cite{CoG,LX,HS}).

\begin{lemma}[Caccioppoli inequality] Let $(G,S)$ be the Cayley graph as (\ref{ABEL}). For any harmonic function $f$ on $B_{6R}(p), R\geq1,$ it holds that
\begin{equation}\label{CCPI}\sum_{x\in B_R(p)}|\nabla f|^2(x)\leq \frac{C}{R^2}\sum_{x\in B_{6R}(p)}f^2(x),\end{equation} where $C=C(S).$
\end{lemma}

The Moser iteration implies the Harnack inequality for positive harmonic functions.
\begin{lemma}[Harnack inequality]\label{HNKI} Let $(G,S)$ be the Cayley graph as (\ref{ABEL}). Then there exist constants $C_1(m,S)$ and $C_2(m,S)$ such that for any $p\in G,$ $R\geq1$ and any positive harmonic function $f$ on $B_{C_1R}(p)$ we have
\begin{equation}\label{HNKI1}\max_{B_R(p)}f\leq C_2 \min_{B_R(p)}f.\end{equation}
\end{lemma}

The mean value inequality follows from one part of Moser iteration (c.f. \cite{CoG,D2,HS}).
\begin{lemma}[Mean value inequality] Let $(G,S)$ be the Cayley graph as (\ref{ABEL}). Then there exists a constant $C_1(m,S)$ such that for any $p\in G,$ $R>0$ and any nonnegative subharmonic function $f$ on $B_R(p)$ we have
\begin{equation}\label{MVI}f(p)\leq\frac{C_1}{|B_R(p)|}\sum_{x\in B_R(p)}f(x).\end{equation}
\end{lemma}

The following Liouville theorem is a corollary of the Harnack inequality (\ref{HNKI1}).
\begin{lemma}[Liouville theorem]\label{LVT} Let $(G,S)$ be the Cayley graph as (\ref{ABEL}). Then any nonnegative harmonic function $f$ on $G$ is constant.
\end{lemma}

Bochner's formula was obtained by Chung-Yau and Lin-Yau on Ricci
flat graphs (\cite{CY,LY}). For the case of Cayley graphs of
finitely generated abelian groups, we present the proof here for the
convenience of readers.
\begin{lemma}[Bochner's formula]\label{BCHNL} Let $(G,S)$ be the Cayley graph as (\ref{ABEL}) and $f$ be a harmonic function defined on $B_1(x),$ for $x\in G$. Then
\begin{equation}\label{BCHN}L^S|\nabla f|^2(x)\geq0.\end{equation}
\end{lemma}
\begin{proof} Let us denote $S=\{s_1,s_2,\cdots,s_{2l}\},$ where $s_i=-s_{i+l}, 1\leq i\leq l.$ Then $$|\nabla f|^2(x)=\sum_{y\sim x}(f(y)-f(x))^2=\sum_{i=1}^{2l}|\delta_{s_i}f|^2(x).$$ Without loss of generality, it suffices to prove $$L^S|\delta_{s_1}f|^2(x)\geq 0.$$
\begin{eqnarray}\label{Bochner1}
L^S|\delta_{s_1}f|^2(x)&=&\sum_{y\sim x}(|\delta_{s_1}f|^2(y)-|\delta_{s_1}f|^2(x))\nonumber\\
&=&\sum_{y\sim x}(f(y+s_1)-f(y))^2-2l|\delta_{s_1}f|^2(x)\nonumber\\
&\geq&\frac{1}{2l}[\sum_{y\sim x}f(y+s_1)-f(y)]^2-2l|\delta_{s_1}f|^2(x)\\
&=&\frac{1}{2l}[\sum_{i=1}^{2l}f(x+s_i+s_1)-\sum_{y\sim x}f(y)]^2-2l|\delta_{s_1}f|^2(x)\nonumber\\
&=&\frac{1}{2l}[\sum_{z\sim (x+s_1)}f(z)-\sum_{y\sim x}f(y)]^2-2l|\delta_{s_1}f|^2(x)\nonumber\\
&=&\frac{1}{2l}[2lf(x+s_1)-2lf(x)]^2-2l|\delta_{s_1}f|^2(x)\\
&=&0,\nonumber
\end{eqnarray} where we use the H\"older inequality in (\ref{Bochner1}) and the harmonicity of $f$ in ($3.8$).
\end{proof}

Combining Bochner's formula with previous results, we obtain Yau's gradient estimate for positive harmonic functions.
\begin{proof}[Proof of Theorem \ref{YGE}] We choose $C_1=7C_1',$ where $C_1'>1$ is the constant $C_1(m,S)$ in Lemma \ref{HNKI}.
Bochner's formula (\ref{BCHN}) implies that $|\nabla f|^2$ is a
subharmonic function on $B_{C_1R}(x).$ Then the theorem follows from
the mean value inequality (\ref{MVI}), the Caccioppoli inequality
(\ref{CCPI}), the volume doubling property (\ref{VDP}) and the
Harnack inequality (\ref{HNKI1}),
\begin{eqnarray*}
|\nabla f|^2(x)&\leq& \frac{C}{|B_R(x)|}\sum_{y\in B_R(x)}|\nabla f|^2(y)\\
&\leq&\frac{C}{R^2|B_R(x)|}\sum_{y\in B_{6R}(x)}f^2(y)\\
&\leq& \frac{C}{R^2|B_{6R}(x)|}\sum_{y\in B_{6R}(x)}f^2(y)\\
&\leq&\frac{C}{R^2}f^2(x).
\end{eqnarray*}
\end{proof}

\begin{corollary}\label{OSCC} Let $(G,S)$ be the Cayley graph as (\ref{ABEL}). There exist constants $C_1(m,S)$ and $C_2(m,S)$ such that for any $x\in G,$ $R\geq1$ any harmonic function $ f$ on $B_{C_1R}(x)$ we have
\begin{equation}\label{OSCE}|\nabla f|(x)\leq \frac{C_2}{R}\mathop{\osc}_{B_{C_1R}(x)} f,\end{equation} where
$\mathop{\osc}_{B_{C_1R}(x)}f:=\mathop{\max}_{B_{C_1R}(x)}f-\min_{B_{C_1R}(x)}f.$
\end{corollary}
\begin{proof} It suffices to choose $f-\min_{B_{C_1R}(x)}f+\epsilon$ (for small $\epsilon$) as the positive harmonic function in Theorem \ref{YGE}.
\end{proof}

\section{Polynomial Growth Harmonic Functions are Polynomials}
The fundamental theorem of finitely generated abelian groups implies
that for any finitely generated abelian group $G$ there exists a
generating set $S^0=\{e_1,\cdots,e_{2m}, w_1,\cdots,w_{2l}\}$
($e_i=-e_{i+m},$ $w_j=-w_{j+l}$, for $1\leq i\leq m,$ $1\leq j\leq
l$) such that $G=G_1\oplus G_2$ is identified with
$\mathds{Z}^m\bigoplus\oplus_{i=1}^l\mathds{Z}_{q_i},$ where
$\{e_1\cdots e_{2m}\}$ generates the torsion-free part
$G_1\cong\mathds{Z}^m$ and $\{w_1\cdots w_{2l}\}$ generates the
torsion part $G_2\cong\oplus_{i=1}^l\mathds{Z}_{q_i}$. For some
fixed $p\in G,$ we denote by $H^d(G,S)=\{u:G\rightarrow
\mathds{R}\mid L^Su=0, \exists C\ s.t.\ |u(x)|\leq
C(d^S(p,x)+1)^d\}$ the space of polynomial growth harmonic functions
on $G$ of growth order less than or equal to $d$. For fixed $S^0,$
$G_1$ is identified with $\mathds{Z}^m,$ then we denote by
$P^d(\mathds{Z}^m):=P^d(\mathds{R}^m)$ the space of polynomials in
$\mathds{R}^m$ restricted to the lattice $\mathds{Z}^m$ with degree
less than or equal to $d.$

First, we consider the easy case that $G$ is torsion-free, i.e. $G\cong\mathds{Z}^m.$ The following theorem generalizes the classical theorem of Heilbronn \cite{H}.
\begin{theorem}\label{TFL} Let $(G,S)$ be the Cayley graph of a torsion-free finitely generated abelian group, $G\cong\mathds{Z}^m.$ Then polynomial growth harmonic functions on $G$ are polynomials, i.e. $$H^d(G,S)\subset P^d(\mathds{Z}^m).$$
\end{theorem}
\begin{proof} Let $S^0=\{e_1,\cdots,e_{2m}\}$ be the standard basis for $\mathds{Z}^m,$ $S=\{s_1,\cdots,s_{2l}\}$ be the generating set for the Cayley graph $(G,S).$
Then $$e_i=\sum_{k=1}^{2l}a_i^ks_k,$$ where $a_i^k\in \mathds{Z}$ for $1\leq i\leq 2m, 1\leq k\leq 2l.$ Let $$\delta_if(x):=\delta_{e_i}f(x)=f(x+e_i)-f(x),$$ for $1\leq i\leq 2m.$
Since $G$ is abelian, it is easy to show that
$$L^S\delta_if=\delta_iL^Sf$$ which implies that $\delta_if$ is harmonic if $f$ is.

We claim that $\delta_if\in H^{d-1}(G,S)$ if $f\in H^d(G,S).$
Although for any $x\in G,$ $x$ and $x+e_i$ may not be neighbors in
the Cayley graph $(G,S),$ there exists a path from $x$ to
$x+e_i=x+\sum_{k=1}^{2l}a_i^ks_k,$ i.e. $x=x_0\sim x_1\sim\cdots\sim
x_t=x+e_i,$ whose length is $t:=\sum_{k=1}^{2l}|a_i^k|\leq
C(S^0,S).$ Note that $f\in H^d(G,S)\Longleftrightarrow
\mathop{\osc}_{B_R(p)}f\leq CR^d$ for some fixed $p\in G$ and
sufficiently large $R\geq R_0.$ For any $x\in G,$ $1\leq i\leq 2m,$
\begin{eqnarray*}
|\delta_if|(x)&=&|f(x+e_i)-f(x)|\\
&\leq&|f(x+e_i)-f(x_{t-1})|+|f(x_{t-1})-f(x_{t-2})|+\cdots+|f(x_1)-f(x)|\\
&\leq&\sum_{j=0}^{t-1}|\nabla f|(x_j)\\
&\leq& \frac{C}{R} \mathop{\osc}_{B_{2R}(x)}f,
\end{eqnarray*}
for $R\geq R_1(S^0,S),$ since $x_j\in B_C(x),$ for $0\leq j\leq t-1.$ The last inequality follows from Corollary \ref{OSCC}. For any $x\in B_R(p),$ we have $B_{2R}(x)\subset B_{3R}(p).$ Then
$$\mathop{\osc}_{B_R(p)}\delta_i f\leq 2 \max_{x\in B_R(p)}|\delta_if|(x)\leq \frac{C}{R}\mathop{\osc}_{B_{3R}(p)}f\leq CR^{d-1},$$ for $R\geq R_1.$ This proves the claim.

Hence by taking finitely many times partial differences and the
Liouville theorem (Lemma \ref{LVT}), we obtain
$$\delta_1^{k_1}\delta_2^{k_2}\cdots\delta_{2m}^{k_{2m}}f=0,$$ for any $k_1+k_2+\cdots+k_{2m}\geq[d]+1$ and $f\in H^d(G,S),$ where $[d]$ is the maximal integer not exceeding $d$. By the basic difference equation theory
or Lemma $2.13$ in Nayar \cite{N}, we conclude that $f$ is a polynomial.
\end{proof}

Then we can prove the generalized Heilbronn's theorem.
\begin{proof}[Proof of Theorem \ref{GHT}]
It suffices to show that $f$ is constant on $G_2,$ i.e.
$f(x+w)=f(x),$ for any $x\in G,$ $w\in G_2.$ An alternative method
by covering argument can be found in the last section.

Let $S^0=\{e_1,\cdots,e_{2m}, w_1,\cdots,w_{2l}\}$ ($e_i=-e_{i+m},$
$w_j=-w_{j+l}$, for $1\leq i\leq m,$ $1\leq j\leq l$) such that
$G=G_1\oplus G_2$ is identified with
$\mathds{Z}^m\bigoplus\oplus_{i=1}^l\mathds{Z}_{q_i},$ where
$\{e_1\cdots e_{2m}\}$ generates the torsion-free part
$G_1\cong\mathds{Z}^m$ and $\{w_j, w_{j+l}\}$ generates
$\mathds{Z}_{q_i}$. Let $\delta_{w_j}f(x):=f(x+w_j)-f(x).$ The same
argument as in the proof Theorem \ref{TFL} implies that
$\delta_{w_j}f\in H^{d-1}(G,S)$ if $f\in H^d(G,S).$ Then
$\delta^{[d]+1}_{w_j}f\equiv0,$ for $f\in H^d(G,S).$ For fixed $x\in
G,$ lifting $\mathds{Z}_{q_j}$ to $\mathds{Z},$ we obtain that $f$
is a polynomial on $\mathds{Z}.$ Hence as a periodic polynomial $f$,
i.e. $f(x+(r+kq_j)w_j)=f(x+rw_j)$ for any $k,r\in\mathds{Z},$ must
be constant, i.e. $f(x+rw_j)=f(x),$ for any $r\in \mathds{Z}.$ Since
this is true for any $1\leq j\leq l,$ we obtain $f(x+w)=f(x),$ for
any $x\in G,$ $w\in G_2.$ Then it is easy to see that
$L^{\pi_{G_1}S}f(x_1)=0$ for any $x_1\in G_1$ if $L^Sf(x)=0$ for any
$x\in G.$ Hence $H^d(G,S)=H^d(G_1, \pi_{G_1}S)\subset
P^{[d]}(\mathds{Z}^m).$
\end{proof}

By the Harnack inequality, we reprove Nayar's theorem.
\begin{proof}[Proof of Theorem \ref{GNT}] It suffices to show $HM^d(G,S)\subset H^d(G,S).$ Without loss of generality, for any harmonic function $f$ satisfying $f(x)\geq -C(d(p,x)+1)^d,$ we need to prove that $f(x)\leq C(d(p,x)+1)^d,$ for some $C.$ For simplicity, we assume $f(p)=0.$ Let $C_1$ be the constant for the Harnack inequality in the Lemma \ref{HNKI}. Then for any $x\in B_R(p),$ $R>0,$ it is easy to see that $B_{C_1R}(x)\subset B_{(C_1+1)R}(p).$ Moreover $$f(y)\geq -C(d(p,y)+1)^d\geq -C((C_1+1)R+1)^d\geq -CR^d,$$ for $y\in B_{(C_1+1)R}(p),$ $R\geq1.$ That is $f(y)+CR^d\geq0$ on $B_{C_1R}(x).$ The Harnack inequality (\ref{HNKI1}) implies that
$$f(x)+CR^d\leq C(f(p)+CR^d)=C_2R^d.$$ Then we have
$$f(x)\leq CR^d,$$ for $x\in B_R(p),$ $R\geq 1.$ Hence there exists a constant $C$ such that $f(x)\leq C(d(p,x)+1)^d.$
\end{proof}

\section{Calculating the Dimension}
For calculating the dimension, by Theorem \ref{GHT}, it suffices to consider harmonic polynomials on a torsion-free finitely generated abelian group. Let $(G,S)$ be the Cayley graph of $G\cong\mathds{Z}^m.$ There exists a generating set $S^0=\{e_1,\cdots,e_{2n}\}$ such that $G$ is identified with $\mathds{Z}^m$ in $\mathds{R}^m.$ For $k\in \mathds{N}\cup\{0\},$ we denote by $P^k(\mathds{R}^m)$ the space of polynomials on $\mathds{R}^m$ of degree less than or equal to $k,$  by $P^k_m:=P^k(\mathds{Z}^m):=\{u:\mathds{Z}^m \rightarrow\mathds{R}\mid u|_{\mathds{Z}^m}=f|_{\mathds{Z}^m}, f\in P^k(\mathds{R}^m)\}$ the space of the restriction of polynomials on $\mathds{R}^m$ to $\mathds{Z}^m$ of degree less than or equal to $k.$ We denote by $R^k_m:=HP^k(\mathds{R}^m)$ the space of harmonic polynomials on $\mathds{R}^m$ ($\Delta u=0$) of degree less than or equal to $k.$ For the Cayley graph $(G,S)$ which is identified with $(\mathds{Z}^m,S),$ we set $D^k_{S,m}:=HP^k(\mathds{Z}^m,S):=\{u\in P^k(\mathds{Z}^m)\mid L^Su=0\}.$ In order to calculate the dimension of discrete harmonic polynomials, we make the dimension comparison between $D^k_{S,m}$ and $R^k_m.$ It is well known for harmonic polynomials on $\mathds{R}^m$ (see \cite{L2}) that
\begin{eqnarray*}\dim R^k_m&=&\binom{m+k-1}{k}+\binom{m+k-2}{k-1},\ \ \ \ \ \ \ \ \ \ (k\geq 1,\ \dim R^0_m=1)\\
\dim R^k_m&=&\dim R^k_{m-1}+\dim R^{k-1}_m,\  \ \ \ \ \ \ \ \ \ \ \ \ \ \ \  \ \ \ \  (k\geq1)\\
\dim P^k_m&=&\dim P^k(\mathds{R}^m)=\sum_{i=0}^k\binom{m+i-1}{i},   \ \ \ \ \  \ \ (k\geq 0)\\
\dim P^k_m&=&\dim R^k_m+\dim P^{k-2}_m.\  \ \ \ \ \ \ \ \ \ \ \ \ \
\ \  \ \  \ \ \ \ \ (k\geq2)\end{eqnarray*} 
\begin{lemma}\label{1dim}\begin{equation}\label{DM1}\dim D^k_{S,1}=\dim R^k_1=2,\end{equation} for any $S,$ $k\geq 1.$
\end{lemma}
\begin{proof}Let $S=\{a_i\}_{i=1}^{2l},$ $a_i\in\mathds{Z},$ and $a_i=-a_{i+l}$ for $1\leq i\leq l.$ Since $S$ is a generating set, at least one $a_i\neq 0.$ For any $f\in D^k_{S,1},$
$f=b_kx^k+b_{k-1}x^{k-1}+\cdots+b_0,$ where $b_0,b_1,\cdots,b_k\in \mathds{R}.$ By  Taylor expansion, the difference equation $L^Sf=0$ reads
$$\sum_{n=1}^{\infty}\sum_{i=1}^l2\frac{a_i^{2n}}{(2n)!}f^{(2n)}(x)=0.$$ Comparing the degree of polynomials, we have
$$(b_kx^k)''=0,$$ for $x\in\mathds{Z}.$ Hence, $k\leq 1,$ i.e. $f$ is linear.
\end{proof}

\begin{lemma}\begin{equation}\label{DMEL}\dim D_{S,m}^k\leq \dim R^k_m,\end{equation} for any $S,$ $k\geq0, m\geq 1.$
\end{lemma}
\begin{proof} It is easy to see that $D^0_{S,m}=R^0_m=\mathrm{const}.$ We apply an induction argument on $k.$ It suffices to prove (\ref{DMEL}) for $k=l$ if it is true for all $k\leq l-1.$
We pick $e_1\in S^0$ and define $\delta_1f(x)=f(x+e_1)-f(x),$ for any function $f$ on $G.$
Since $G$ is abelian, $\delta_1 L^S=L^S \delta_1,$ then $\delta_1$ is a well defined linear operator,
$$\delta_1:D_{S,m}^k\rightarrow D_{S,m}^{k-1}.$$
By linear algebra, \begin{equation}\label{LAD}\dim D^k_{S,m}=\dim \ker\delta_1+\dim \mathrm{im}\ \delta_1,\end{equation} where $\ker\delta_1$ and $\mathrm{im}\ \delta_1$ are Kernel and Image of $\delta_1.$
There is a natural projection $P:\mathds{Z}^m\rightarrow\mathds{Z}^{m-1},$ for any $x=(x_1,x_2,\cdots,x_m)\in\mathds{Z}^m,$ $Px=x',$ where $x'=(x_2,\cdots,x_m).$ For any generating set $S$ for $\mathds{Z}^m,$ $S'=\{s'\mid s\in S\}$ is the generating set of $\mathds{Z}^{m-1}.$ Moreover, $S^{(t)}:=P^tS$ is the generating set for $\mathds{Z}^{m-t},$ $1\leq t\leq m-1.$

For any $f\in \ker\delta_1,$ i.e. $\delta_1f=0,$ then $f(x_1,x_2,\cdots,x_m)=g(x_2,\cdots,x_m).$ Hence
\begin{eqnarray*}0&=&L^Sf(x)=\sum_{s\in S}(f(x+s)-f(x))\\
&=&\sum_{s\in S}(g(x'+s')-g(x'))=L^{S'}g(x').
\end{eqnarray*}
That is $\ker\delta_1=D^k_{S',m-1}.$
Hence (\ref{LAD}) implies that $$\dim D^k_{S,m}\leq\dim D^k_{S',m-1}+\dim D^{k-1}_{S,m},$$ for any $S,$ $k\geq1,$ $m\geq1.$
Then it follows that
\begin{eqnarray*}
\dim D^l_{S,m}&\leq&\dim D^l_{S',m-1}+\dim D^{l-1}_{S,m}\\
&\leq&\dim D^l_{S'',m-2}+\dim D^{l-1}_{S',m-1}+\dim D^{l-1}_{S,m}\\
&\leq&\cdots\cdots\\
&\leq&\dim D^l_{S^{(m-1)},1}+\sum_{i=2}^m\dim D_{S^{(m-i)},i}^{l-1}\\
&\leq&2+\sum_{i=2}^m \dim R_i^{l-1}\\
&=&\dim R_1^{l}+\sum_{i=2}^m \dim R_i^{l-1}\\
&=&\dim R^l_m,
\end{eqnarray*}
where we use Lemma $\ref{1dim},$ the inductive assumption (\ref{DMEL}) for $k\leq l-1$ and some facts in $\mathds{R}^n$.

\end{proof}

Now we can prove the main Theorem \ref{DCT}.
\begin{proof}[Proof of Theorem \ref{DCT}] It suffices to show that $$\dim D^k_{S,m}=\dim R^k_m,$$ for any $S,$ $k\geq0$ and $m\geq1.$ We may assume $k\geq2,$ otherwise it is trivial.
By $S=-S,$ the Laplacian operator $L^S$ is a linear operator,
\begin{equation}\label{eq12}L^S:P^k_m\rightarrow P^{k-2}_m.\end{equation} Since $\ker L^S=D^k_{S,m},$
we have
\begin{eqnarray*}
\dim P^k_m&=&\dim \ker L^S+\dim \mathrm{im}\ L^S\\
&\leq&\dim D^k_{S,m}+\dim P^{k-2}_m\\
&\leq&\dim R^k_m+\dim P^{k-2}_m\\
&=&\dim P^{k}_m,
\end{eqnarray*} which follows from (\ref{DMEL}).
Hence $\dim D^k_{S,m}=\dim R^k_m.$
\end{proof}

\begin{remark} \emph{In the above theorems, all the inequalities for
the dimension comparison are actually equalities. Hence, we obtain
that}
  $\delta_1$ \emph{and} $L^S$ \emph{are surjective linear
    operators. In fact, Heilbronn \cite{H} used the technical lemma in
    difference equation theory (see \cite{Hob,DS}) that}
  $L^{S^0}:P^k_m\rightarrow P^{k-2}_m$ \emph{is surjective to
    calculate the dimension. Conversely, by the dimension comparison, we
    obtain a more general difference equation lemma.}
\end{remark}

\begin{corollary}\label{coro1} Let $(\mathds{Z}^m,S)$ be the Cayley graph of $\mathds{Z}^m.$ Then
$$\delta_1:D^k_{S,m}\rightarrow D^{k-1}_{S,m},$$ and
\begin{equation}\label{eq11}L^S:P^k_m\rightarrow P^{k-2}_m\end{equation} are surjective linear
operators.
\end{corollary}

\section{Dimension of Higher Order Harmonic Functions}
Let $(G,S)$ be the Cayley graph of the finitely generated abelian
group with a finite generating set $S,$
$G\cong\mathds{Z}^m\bigoplus\oplus_{i=1}^l\mathds{Z}_{q_i}.$ We
define the $n$th Laplace operator as
$$L^{n,S}=L^S\circ L^S\circ\cdots\circ L^S,$$ i.e. $n$-times
composition of Laplace operators. A function $u$ on $G$ is called
$n$-harmonic if $L^{n,S}u=0.$ We use the results for Laplace
operators in Theorem \ref{GHT} and Corollary \ref{coro1} to obtain
the dimension estimate for higher order harmonic functions.

\begin{proof}[Proof of Theorem \ref{hoo}] To simplify the argument,
we only prove the theorem for $n=2$ which directly applies to
general $n$. For the first assertion, it suffices to show that
$H^{2,d}(G,S)\subset P^d(\Z^m)$ for any torsion-free
$G\cong\mathds{Z}^m$ (as in Theorem \ref{TFL}) since we may reduce
the general case to this one. For a general $G=G_1\oplus
G_2\cong\mathds{Z}^m\bigoplus\oplus_{i=1}^l\mathds{Z}_{q_i}$ and a
generating set $S,$ we may find a torsion-free
$\tilde{G}\cong\Z^{m+l}\cong\Z^m\oplus\Z^l$ and a generating set
$\tilde{S}$ such that the Cayley graph $(\tilde{G},\tilde{S})$ is a
natural covering of $(G,S)$. For any $f\in H^{2,d}(G,S),$ the
lifting function $\tilde{f}$ of $f$ on $\tilde{G}$ is $2$-harmonic
and of polynomial growth in the sense of $(\tilde{G},\tilde{S}).$
This implies that $\tilde{f}$ is a polynomial on $\tilde{G}$ by the
torsion-free case. Since $\tilde{f}$ is periodic in the factor
$\Z^l$ by the lifting, this yields that as a polynomial $\tilde{f}$
is constant on $\Z^l.$ Hence $f$ is a polynomial on $\Z^m$ and
constant on the $G_2$ factor.

From now on, we assume $G\cong \Z^m.$ For any $f\in
H^{2,d}(\Z^m,S),$
$$L^{2,S}f=L^S(L^Sf)=0.$$ Since $|L^Sf(x)|\leq m\max_{y\sim
x}\{|f(y)|+|f(x)|\},$ we have $L^Sf\in H^{d}(\Z^m,S).$ By Theorem
\ref{GHT} or \ref{TFL}, we know that $L^Sf$ is a polynomial on
$\Z^m$, denoted by $g:=L^Sf.$ By \eqref{eq11} in Corollary
\ref{coro1}, we may find a polynomial, say $f_1,$ satisfying
$L^Sf_1=g.$ Hence $L^S(f-f_1)=0.$ The polynomial growth of $f-f_1$
implies that $f-f_1$ is actually a polynomial. This proves that $f$
is a polynomial (In fact, $f$ is a polynomial of degree less than or
equal to $d$ by $f\in H^{2,d}(\Z^m,S)$).

Now we calculate the dimension of $H^{2,d}(\Z^m,S).$ We set $k=[d].$
By \eqref{eq12}, we have a linear operator,
$$L^{2,S}:P_m^k\rightarrow P_m^{k-4}.\ \ \ \ \ (k\geq 4)$$
It is easy to see that $\ker L^{2,S}=H^{2,k}(\Z^m,S).$ Moreover,
\eqref{eq11} implies that $L^{2,S}$ is surjective. Hence $$\dim
H^{2,k}(\Z^m,S)=\dim \ker L^{2,S}=\dim P_m^k-\dim P_m^{k-4}.$$ This
proves the theorem.
\end{proof}


\begin{thebibliography}{00}

\bibitem{Al3} \textit{G. Alexopoulos}, An application of homogenization theory to harmonic analysis: Harnack inequalities and Riesz transforms on Lie groups of polynomial growth, Canad. J. Math. \textbf{44} (1992), no. 4, 691-727.


\bibitem{Al2} \textit{G. Alexopoulos}, An application of homogenization theory to harmonic analysis on solvable Lie groups of polynomial growth, Pacific J. Math. \textbf{159} (1993), no. 1, 19-45.


\bibitem{Al1}
\textit{G. Alexopoulos}, Random walks on discrete groups of polynomial volume growth, Ann. Probab. \textbf{30} (2002), no. 2, 723-801.

\bibitem{ALO} \textit{G. Alexopoulos} and \textit{N. Lohou\'e}, Sobolev inequalities and harmonic functions of polynomial growth, J. London Math. Soc. (2) \textbf{48} (1993), no. 3, 452-464.




\bibitem{AvL} \textit{M. Avellaneda} and \textit{F. Lin},
Un th\'eor\`eme de Liouville pour des \'equations elliptiques \`a coefficients p\'eriodiques,
C. R. Acad. Sci. Paris S\'er. I Math. \textbf{309} (1989), no. 5, 245-250.


\bibitem{BBI}\textit{D. Burago, Yu. Burago} and \textit{S. Ivanov}, A course in metric geometry, Graduate Studies in Mathematics \textbf{33},
American Mathematical Society, Providence, RI, 2001.

\bibitem{CW} \textit{R. Chen} and \textit{J. Wang}, Polynomial growth solutions to higher-order linear elliptic equations and systems, Pacific J. Math. \textbf{229} (2007), no. 1, 49-61.

\bibitem{CheY} \textit{S. Y. Cheng} and \textit{S. T. Yau}, Differential equations on Riemannian manifolds and their geometric applications,
Comm. Pure Appl. Math. \textbf{28} (1975), 333-354.


\bibitem{CY} \textit{F. R. K. Chung} and \textit{S. T. Yau}, Logarithmic Harnack inequalities,
Math. Res. Lett. \textbf{3} (1996), no. 6, 793-812.

\bibitem{CM1} \textit{T. H. Colding} and \textit{W. P. Minicozzi II}, Harmonic functions with polynomial growth, J. Diff. Geom. \textbf{46} (1997), no. 1, 1-77.


\bibitem{CM2} \textit{T. H. Colding} and \textit{W. P. Minicozzi II}, Harmonic functions on manifolds, Ann. of Math. (2) \textbf{146} (1997), no.
3, 725-747.


\bibitem{CM3} \textit{T. H. Colding} and \textit{W. P. Minicozzi II}, Weyl type bounds for harmonic functions, Invent. Math. \textbf{131} (1998), no. 2, 257-298.

\bibitem{CM4} \textit{T. H. Colding} and \textit{W. P. Minicozzi II}, Liouville theorems for harmonic sections and applications, Comm. Pure Appl. Math. \textbf{51} (1998), no. 2, 113-138.

\bibitem{CoG} \textit{T. Coulhon} and \textit{A. Grigoryan}, Random walks on graphs with regular volume growth, Geom. Funct. Anal. \textbf{8} (1998), no. 4, 656-701.


\bibitem{CS}  \textit{T. Coulhon} and \textit{L. Saloff-Coste}, Vari\'et\'es riemanniennes isom\'etriques \`a l'infini, Rev. Mat. Iberoamericana \textbf{11} (1995), no. 3, 687-726.

\bibitem{D1} \textit{T. Delmotte}, Harnack inequalities on graphs, S\'eminaire de Th\'eorie Spectrale et G\'eom\'etrie, Vol. \textbf{16}, Ann\'ee 1997-1998, 217-228.

\bibitem{D2} \textit{T. Delmotte}, In\'egalit\'e de Harnack elliptique sur les graphes, Colloq. Math. \textbf{72} (1997), no. 1, 19-37.

\bibitem{DF} \textit{H. Donnelly} and \textit{C. Fefferman},
Nodal domains and growth of harmonic functions on noncompact manifolds,
J. Geom. Anal. \textbf{2} (1992), no. 1, 79-93.

\bibitem{DS} \textit{R. J. Duffin} and \textit{E. P. Shelly}, Difference equations of polyharmonic type, Duke Math. J. \textbf{25} (1958) 209-238.

\bibitem{G} \textit{M. Gromov}, Groups of polynomial growth and expanding maps, Inst. Hautes \'Eudes Sci. Publ. Math. No. \textbf{53} (1981), 53-73.

\bibitem{HS} \textit{I. Holopainen} and \textit{P. M. Soardi}, A strong Liouville theorem for $p$-harmonic functions on graphs, Ann. Acad. Sci. Fenn. Math. \textbf{22} (1997), no. 1, 205-226.

\bibitem{H} \textit{H. A. Heilbronn}, On discrete harmonic functions, Proc. Cambridge Philos. Soc. \textbf{45}, (1949). 194-206.

\bibitem{Hob} \textit{E. W. Hobson},
The theory of spherical and ellipsoidal harmonics, Chelsea Publishing Company, New York, 1955.

\bibitem{Hu1} \textit{B. Hua}, Generalized Liouville theorem in nonnegatively curved Alexandrov spaces, Chin. Ann. Math. Ser. B \textbf{30} (2009), no. 2, 111-128.


\bibitem{Hu2} \textit{B. Hua}, Harmonic functions of polynomial growth on singular spaces with nonnegative Ricci curvature, Proc. Amer. Math. Soc. \textbf{139} (2011), 2191-2205.

\bibitem{HJ} \textit{B. Hua} and \textit{J. Jost}, Polynomial growth harmonic functions on groups of polynomial volume growth, preprint, arXiv:1201.5238.


\bibitem{HJL} \textit{B. Hua, J. Jost} and \textit{S. Liu}, Geometric aspects of infinite semiplanar graphs with nonnegative curvature, preprint, arXiv:1107.2826.


\bibitem{KLe} \textit{S. W. Kim} and \textit{Y. H. Lee}, Polynomial growth harmonic functions on connected sums of complete Riemannian manifolds, Math. Z. \textbf{233} (2000), no. 1, 103-113.

\bibitem{Kl} \textit{B. Kleiner}, A new proof of Gromov's theorem on groups of polynomial growth, J. Amer. Math. Soc. \textbf{23} (2010), no. 3, 815-829.

\bibitem{KP} \textit{P. Kuchment} and \textit{Y. Pinchover}, Liouville theorems and spectral edge behavior on abelian coverings of compact manifolds, Trans. Amer. Math. Soc. \textbf{359} (2007), no. 12, 5777-5815.


\bibitem{M} \textit{J. Milnor}, A note on curvature and fundamental group,
J. Differential Geometry \textbf{2} (1968), 1-7.

\bibitem{MS} \textit{J. Moser} and \textit{M. Struwe}, On a Liouville-type theorem for linear and nonlinear elliptic differential equations on a torus,
Bol. Soc. Brasil. Mat. (N.S.) \textbf{23} (1992), no. 1-2, 1-20.

\bibitem{La} \textit{G. F. Lawler}, Intersections of random walks, Probability and its Applications, Birkh\"auser Boston Inc., Boston, MA, 1991.

\bibitem{Le} \textit{Y. H. Lee}, Polynomial growth harmonic functions on complete Riemannian manifolds, Rev. Mat. Iberoamericana \textbf{20} (2004), no. 2, 315-332.

\bibitem{L1}  \textit{P. Li}, Harmonic sections of polynomial growth, Math. Res. Lett. \textbf{4} (1997), no. 1, 35-44.


\bibitem{L2}  \textit{P. Li}, Harmonic functions and applications to complete manifolds (lecture notes), preprint.

\bibitem{LT} \textit{P. Li} and \textit{L.-F. Tam}, Complete surfaces with finite total curvature,
J. Differential Geom. \textbf{33} (1991), no. 1, 139-168.

\bibitem{LW1} \textit{P. Li} and \textit{J. Wang}, Mean value inequalities, Indiana Univ. Math. J. \textbf{48} (1999), no. 4, 1257-1283.

\bibitem{LW2} \textit{P. Li} and \textit{J. Wang}, Counting dimensions of L-harmonic functions, Ann. of Math. (2) \textbf{152} (2000), no. 2, 645-658.

\bibitem{LW3} \textit{P. Li} and \textit{J. Wang}, Polynomial growth solutions of uniformly elliptic operators of non-divergence form, Proc. Amer. Math. Soc. \textbf{129} (2001), no. 12, 3691-3699.


\bibitem{Lin} \textit{F. Lin}, Asymptotically conic elliptic operators and Liouville type theorems. Geometric analysis and the calculus of variations, 217-238, Int. Press, Cambridge, MA, 1996.


\bibitem{LX} \textit{Y. Lin} and \textit{L. Xi}, Lipschitz property of harmonic function on graphs, J. Math. Anal. Appl. \textbf{366} (2010), no. 2, 673-678.

\bibitem{LY} \textit{Y. Lin} and \textit{S. T. Yau}, Ricci curvature and eigenvalue estimate on locally finite graphs, Math. Res. Lett. \textbf{17} (2010), no. 2, 343-356.



\bibitem{N} \textit{P. Nayar}, On polynomially bounded harmonic functions on the $\mathds{Z}^d$ lattice, Bull. Pol. Acad. Sci. Math. \textbf{57} (2009), no. 3-4, 231-242.

\bibitem{Ra} \textit{M. S. Raghunathan}, Discrete subgroups of Lie groups, Ergebnisse der Mathematik und ihrer Grenzgebiete, Band 68, Springer-Verlag, New York-Heidelberg, 1972.

\bibitem{R} \textit{D. J. S. Robinson}, A course in the theory of groups, second edition, Graduate Texts in Mathematics, 80. Springer-Verlag, New York, 1996.



\bibitem{ST} \textit{Y. Shalom} and \textit{T. Tao}, A finitary version of Gromov's polynomial growth theorem, Geom. Funct. Anal. \textbf{20} (2010), no. 6, 1502-1547.


\bibitem{STW} \textit{C.-J. Sung, L.-F. Tam} and \textit{J. Wang}, Spaces of harmonic functions, J. London Math. Soc. (2) \textbf{61} (2000), no. 3, 789-806.

\bibitem{Su} \textit{M. Suzuki}, Group theory I, Translated from the Japanese by the author, Grundlehren der Mathematischen Wissenschaften [Fundamental Principles of Mathematical Sciences], 247. Springer-Verlag, Berlin-New York, 1982.

\bibitem{T} \textit{L.-F. Tam}, A note on harmonic forms on complete manifolds, Proc. Amer. Math. Soc. \textbf{126} (1998), no. 10, 3097-3108.

\bibitem{W}
\textit{J. Wang}, Linear growth harmonic functions on complete manifolds, Comm. Anal. Geom.
\textbf{4} (1995), 683-698.


\bibitem{Y1}
\textit{S. T. Yau}, Harmonic functions on complete Riemannian manifolds,
Comm. Pure Appl. Math. \textbf{28} (1975), 201-228.

\bibitem{Y2}
\textit{S. T. Yau}, Nonlinear analysis in geometry, Enseign. Math.
\textbf{33} (1987), (2), 109-158.


\bibitem{Y3}
\textit{S. T. Yau}, Differential Geometry: Partial Differential Equations on
Manifolds, Proc. of Symposia in Pure Mathematics, \textbf{54}, part 1, Ed. by
R.Greene and S.T. Yau, 1993.




\end{thebibliography}
\end{document}